\documentclass[12pt]{amsart}

\usepackage{amssymb}
\usepackage{graphicx}
\usepackage{caption} 

\usepackage{amsmath}
\usepackage{amsfonts}

\usepackage{amsthm}

\usepackage{enumerate}
\usepackage{marginnote}
\reversemarginpar
\usepackage[small,nohug,heads=vee]{diagrams}
\diagramstyle[labelstyle=\scriptstyle]

\makeatletter
\@namedef{subjclassname@2010}{%
  \textup{2010} Mathematics Subject Classification}
\makeatother

\newtheorem{thm}{Theorem}[section]

\newtheorem{lem}[thm]{Lemma}
\newtheorem{prop}[thm]{Proposition}

\newtheorem{mainthm}[thm]{Main Theorem}

\theoremstyle{definition}
\newtheorem{defin}[thm]{Definition}
\newtheorem{rem}[thm]{Remark}

\newtheorem*{xrem}{Remark}

\numberwithin{equation}{section}

\def\cA{\mathcal{A}}
\def\C{\mathbb{C}}
\def\Chat{\hat{\mathbb C}}
\def\D{\mathbb{D}}
\def\S{\mathbb{S}}
\def\Z{\mathbb{Z}}

\thanks{Research was partially supported by the ERC advanced grant ``HOLOGRAM", the Deutsche Forschungsgemeinschaft SCHL 670/4, and by the grant 346300 for IMPAN from the Simons Foundation and the matching 2015-2019 Polish MNiSW fund}

\begin{document}

\title[Newton maps and parabolic surgery]{Newton maps of complex exponential functions and parabolic surgery}

\author[K. Mamayusupov]{Khudoyor Mamayusupov}
\address{National Research University Higher School of Economics, Faculty of Mathematics, Usacheva 6, Moscow, Russia}
\address{Jacobs University Bremen\\
Campus Ring 1, 28759, Bremen, Germany}
\address{Institute of Mathematics Polish Academy of Sciences, \' Sniadeckich 8, 00-656 Warsaw, Poland}
\email{kmamayusupov@hse.ru}
\date{\today}
\begin{abstract} The paper deals with Newton maps of complex exponential functions and a surgery tool developed by P. Ha\"issinsky. The concept of \emph{``Postcritically minimal"} Newton maps of complex exponential functions are introduced, analogous to postcritically finite Newton maps of polynomials. The dynamics preserving mapping is constructed between the space of postcritically finite Newton maps of polynomials and the space of postcritically minimal Newton maps of complex exponential functions.
\end{abstract}



\maketitle
\section{Introduction}
Let $f : \C \to \C$ be an entire function. The meromorphic function defined by $N_f(z):=z-f(z)/f'(z)$ is the {\em Newton map} of $f$; we also say that $N_f$ is obtained by applying Newton's method $f$. We define a \emph{complex exponential function} to be an entire function of the form
\begin{equation}
\label{Eq:peq}
p(z) e^{q(z)},
\end{equation}
for $p=p(z)$ and $q=q(z)$ polynomials. Newton's method applied to \eqref{Eq:peq} yields $z-\frac{p(z)}{p'(z)+p(z) q'(z)}$, which is a rational function since the exponential terms $e^q$ cancel. 

This paper deals with the family of rational functions obtained by applied Newton's method to the family of complex exponential functions \eqref{Eq:peq}. Finite fixed points of Newton maps are attracting, the roots of polynomial $p$, and the point at $\infty$ is a parabolic fixed point when $\deg (q)>0$ and a repelling fixed point otherwise.
The main goal of this paper is to lay the groundwork to the theory of all Newton maps of complex exponential functions that runs parallel to the theory of Newton maps of polynomials. 

Recall that a rational function is called postcritically finite if its critical orbits form a finite set. Every critical point in the Fatou set of a postcritically finite rational function eventually lands at superattracting periodic points. We say that a pair of distinct critical points $c_1$ and $c_2$ of a rational function $R$ are in \emph{critical orbit relations} if $R^{\circ k}(c_1)=R^{\circ l}(c_2)$ for integers $k$ and $l$.

\begin{defin}[Minimal critical orbit relations ]\label{Def:minimal-rel} Let $c_1\in U_1$ and $c_2\in U_2$ be critical points of a rational function $R$ of degree at least $2$, where $U_1$ and $U_2$ are connected components of the Fatou set of $R$. Assume $R^{\circ m}(U_1)=U_2$ with minimal $m\ge1$. We say that $c_1$ and $c_2$ are in \emph{minimal critical orbit relations} if $R^{\circ m}(c_1)=c_2$.
\end{defin}
Let a critical point $c_1\in U_1$ be captured by other critical point $c_2\in U_2$ with $R^{\circ m}(U_1)=U_2$ with minimal $m\ge1$, where $U_1$ and $U_2$ are connected components of the Fatou set of $R$. Of course $c_1$ needs at least $m$ iterates to land at $c_2$. The minimality condition requires that $c_1$ is not allowed to take more iterates than $m$ when it lands at $c_2$ for the first time. 

Let us introduce a new notion, and call it postcritically minimal, in the family of Newton maps of \emph{complex exponential functions}.

\begin{defin}[Postcritically minimal Newton map]\label{Def:PCM} A Newton map $N_{pe^{q}}$ is called postcritically minimal (PCM) if its Fatou set consists of superattracting basins and parabolic basins of $\infty$, and the following holds:
\begin{enumerate}[(a)]
\item critical orbits in the Julia set and in superattracting basins are finite;
\item every immediate basin of $\infty$ contains one (possibly with high multiplicity) critical point, and all other critical points in basin of $\infty$ are in \emph{minimal critical orbit relations} such that if $c$ is a critical point in a strictly preperiodic component $U$, of preperiod $m\ge 1$, of the basin of $\infty$, then $N^{\circ m}_{pe^{q}}(c)$ is a critical point in one of immediate basins of $\infty$. 
\end{enumerate}
\end{defin}

In the parameter space of Newton maps of polynomials (parametrized by locations of roots of polynomials) connected components of hyperbolic functions are called hyperbolic components. Every bounded hyperbolic component contains a unique ``center", which is known to be a postcritically finite function \cite{Hypcom}. Let $p$ be a polynomial with simple roots, $\deg(p)=d-n\ge 0$, and $\deg(q)=n\ge1$. Similarly, we define the \emph{stable} components of the parameter plane of Newton maps of \eqref{Eq:peq}, parametrized by coefficients of $p$ and $q$, to be connected components consisting of functions whose all critical points belong to attracting basins or parabolic basin of $\infty$. For this family of functions the point at $\infty$ is persistently parabolic with the multiplier $+1$, thus these stable functions form an open set in the parameter plane. Similar to hyperbolic components, one can observe that using a suitable surgery developed by C.~McMullen \cite{McM3} every function within a bounded stable component can be quasiconformally perturbed to a ``center'' one, which is a \emph{postcritically minimal} Newton map. This process keeps the dynamics unchanged on Julia sets. Not every postcritically minimal Newton map of \eqref{Eq:peq} is a ``center", likewise not every postcritically finite Newton map of polynomial is a ``center" of hyperbolic component in the corresponding parameter space.

The definition of postcritical minimality is a generalization of postcritically finiteness by allowing parabolic domains with minimal critical orbit relations. Finding an appropriate and compact definition for postcritical minimality is motivated by the above mentioned surgery of C.~McMullen. The postcritical minimality condition is automatically satisfied for \emph{postcritically finite} Newton maps of polynomials. In superattracting domains a captured critical point (free critical point) eventually lands at a critical cycle in minimal iterate. In contrast, in parabolic domains a captured critical point may take longer iterate (not minimal) before it lands at the critical point in one of immediate basins of a parabolic periodic point.

In Proposition~\ref{Thm:Characterization_PCM} it is shown that every component of the Fatou set of postcritically minimal Newton maps contains a unique center, defined below, and the map has normal forms: in immediate basins of parabolic fixed point the dynamics are conjugate to parabolic Blaschke products $P_k(z):=\frac{z^k+a}{1+a z^k}$, where $a=\frac{k-1}{k+1}$, and in all other Fatou components the dynamics are conjugated to power maps $z\mapsto z^k$, where $k$ is the local degree of the function at the center of the Fatou component.

\begin{defin}[Centers of Fatou components]\label{Def:Center}
Let $N_{pe^{q}}$ be a PCM Newton map and $U$ be a component of its Fatou set. A \emph{center} of $U$ is a point $\xi\in U$ such that
\begin{enumerate}[(a)]
\item when $U$ is a component of a superattracting basin, then $\xi$ is its unique critical periodic point;
\item when $U$ is an immediate basin of $\infty$, then $\xi$ is its unique critical point;
\item when $U$ is strictly preperiodic, of preperiod $m_U\ge 1$, then $N_{pe^{q}}^{\circ m_U}(\xi)$ is the center of $N_{pe^{q}}^{\circ m_U}(U)$, which is an immediate basin of a superattracting periodic point or of the basin of $\infty$.
\end{enumerate}
\end{defin}

Let $N_p$ be the postcritically finite Newton map of a polynomial $p$ with $d=\deg(p)\ge 3$. Then $\deg(N_p)=d$ and every root $\xi$ of the polynomial $p$ is simple. As a rational function of degree $d\ge3$, the Newton map $N_p$ has $2d-2$ critical points counting with multiplicities, $d$ of them are superattracting fixed points (without counting multiplicities), the roots of $p$. All other $d-2$ or less critical points have finite orbits, some are on the Julia set. Critical points on the Fatou set may form superattracting periodic points of high period, these in turn may capture some other critical points. But always, all critical orbit relations in the Fatou set are \emph{minimal} critical orbit relations.

Let $N_{pe^{q}}$ be a postcritically minimal Newton map with $\deg(N_{pe^{q}})=d\ge 3$ and $\deg(q)=n\ge 1$. Then $p$ is not constantly zero and $\deg(p)=d-n$. When $\deg(p)\ge1$ the roots of $p$ are simple and are the only superattracting fixed points of $N_{pe^{q}}$. There are $n$ immediate basins of the parabolic fixed point at infinity with exactly one critical point (possibly with high multiplicity) in each. All other $d-2$ or less critical points of $N_{pe^{q}}$ eventually land at repelling periodic points on the Julia set, thus have finite orbits. Some of the other critical points on the Fatou set may form their own superattracting periodic points of high period, and all are always in minimal critical orbit relations among themselves. The critical orbit relations could be either on the superattracting domains or on the parabolic domains at infinity. 

In \cite{Ha}, P. Ha\"issinsky developed a surgery tool to turn an attracting domain of a polynomial to a parabolic domain of other polynomial while keeping the dynamics on the Julia set unchanged. This procedure is referred to ``parabolic surgery". It takes a polynomial with an attracting fixed point, an accessible repelling fixed point on the boundary of the immediate basin of attraction, and changes this domain to a parabolic one; the repelling fixed point becomes a parabolic fixed point for the new polynomial, but the dynamics will not be changed on the Julia set. 

For the Newton map of a polynomial the basins of attraction of fixed points have a common boundary point at $\infty$, the parabolic surgery can be applied to make the point at $\infty$ a parabolic fixed point for the resulting map. After a suitable normalization, an affine conjugacy, the surgery process results to a rational function which is the Newton map of $p e^q$ for some polynomials $p$ and $q$. This operation defines a mapping from the space of all postcritically finite Newton maps of \emph{polynomials} to the space of Newton maps of \eqref{Eq:peq}. This mapping preserves dynamics, via a David homeomorphism, away from marked basins, in particular, dynamics are conjugate on the Julia sets. Moreover, the marked basins for doing surgery become parabolic basins. Marked access of marked immediate superattracting basins, in turn, become dynamical accesses in the parabolic immediate basins of the resulting map. In \cite{Ma}, we give a one-to-one correspondence between the space of postcritically minimal Newton maps of \eqref{Eq:peq} and the space of postcritically finite Newton maps of polynomials. This will be a classification of Newton maps of complex exponential functions by complementing the classification result for Newton maps of polynomials \cite{LMS1, LMS2}. Our main theorem is the following.
\begin{mainthm}[Parabolic surgery for Newton map of polynomial]
\label{Thm:Parabolic_Surgery_Newton}
Let $N_p$ be a postcritically finite Newton map of degree $d\ge 3$, and $\Delta^+_n$ be its marked channel diagram with $1\le n\le d$. For all $1\le j\le n$, let $\cA(\xi_j)$ be the marked basins of superattracting fixed points $\xi_j$. Then there exist a homeomorphism $\phi$ and a postcritically minimal Newton map $N_{\tilde p e^{\tilde q}}$ of degree $d$ with $\deg( \tilde q)=n$ such that
\begin{enumerate}[(a)]
\item $\phi\circ N_p(z)=N_{\tilde p e^{\tilde q}}\circ \phi(z)$ for all $z\not \in \bigcup_{1\leq j\leq n}\cA^{\circ}(\xi_j)$; in particular, $\phi:J(N_p)\to J(N_{\tilde p e^{\tilde q}})$ is a homeomorphism which conjugates $N_p$ to $N_{\tilde p e^{\tilde q}}$;
\item $\phi(\infty)=\infty$, and $\phi(\bigcup_{1\leq j\leq n}\cA(\xi_j))$ is the full basin of the parabolic fixed point at $\infty$ of $N_{\tilde p e^{\tilde q}}$;
\item $\phi$ is conformal in the interior of $\Chat\setminus \bigcup_{1\leq j\leq n}\cA(\xi_j)$;
\item the marked invariant accesses of the marked channel diagram $\Delta^+_n$ of $N_p$ correspond to all dynamical accesses of the parabolic basin of $\infty$ for $N_{\tilde p e^{\tilde q}}$.
\end{enumerate}
\end{mainthm}

\begin{xrem}
	In the theorem the conjugacy breaks down only on the marked immediate basins.
\end{xrem}

The class of Newton maps of complex exponential functions (for $\deg (q)\ge1$) was studied very little; in \cite{Haruta}, in a pioneering work, M.~Haruta showed that the area of every immediate basin of an attracting fixed point is finite if $\deg (q)\geq 3$. This is in contrast to the corresponding result for Newton maps of polynomials where the area of every immediate attracting basin is infinite. In certain cases the area is infinite even when $\deg (q) \in \{1, 2\}$, see \cite{Cil, CJ}. 

\emph{Overview of the paper.} In Section~\ref{Sec:Dynamics}, we review the dynamical properties of Newton maps and state Proposition~\ref{Thm:access} on number of invariant accesses to $\infty$. Structure of postcritically minimal Newton maps on each components of the Fatou set is provided in Proposition~\ref{Thm:Characterization_PCM}. In Section~\ref{Sec:Haissinsky}, we review the whole construction of the proof of Ha\"issinsky theorem (parabolic surgery) for polynomials and adapt it specifically for Newton maps of polynomials. In the last section we prove Main Theorem~\ref{Thm:Parabolic_Surgery_Newton}.

\emph{Notations}. Let us denote by $R^{\circ n}$ the $n^{\text{th}}$ iterate of a rational function $R$ and $z^n$ denotes the $n^{\text{th}}$ power of a number $z$. For a rational function $R$, let $F(R)$ and $J(R)$ denote the Fatou set and the Julia set of $R$ respectively. For a polynomial $p$, denote by $K(p)$ the filled Julia set of $p$. Denote the local degree of a function $R$ at a point $z$ by $\deg(R,z)$.	 Set $C_R = \{z|\deg(R,z)>1 \}$, critical points of $R$, and $P_R=\overline{\bigcup_{n\ge1} R^{\circ n}(C_R)}$, the postcritical set of $R$. The function $R$ is called \emph{postcritically finite} (\emph{PCF}) if $P_R$ is finite. The function $R$ is called \emph{geometrically finite} if the intersection $P_R\cap \nolinebreak J(R)$ is finite. 

\section{Dynamical properties of Newton maps}\label{Sec:Dynamics}

Let $f:\C\to\C$ be an entire function (polynomial or transcendental entire function). The meromorphic function given by $N_f(z):=z-f(z)/f'(z)$ is called the Newton map of $f(z)$.
The following theorem describes a Newton map in terms of its fixed point multipliers.

\begin{thm}\cite{RS}
\label{Thm:Theorem_Newton} Let $N:\C \to \Chat$ be a meromorphic function. It is the Newton map of an entire function $f:\C\to\C$ if and only if for each fixed point $\xi$, $N(\xi)=\xi$, there is a natural number $m=m_{\xi} \in \mathbb N$ such that $N'(\xi)=(m-1)/m$. In this case there exists a constant $c\in \C \backslash \{0\}$ such that $f=c e^{\int\frac{d \zeta}{\zeta-N(\zeta)} }$. Two entire functions $f$ and $g$ have the same Newton map if and only if $f=c g$ for some constant $c\in \C \backslash \{0\}$.
\end{thm}

From \cite{RS}, it is known that for an entire function $f:\C\to\C$ its Newton map $N_f$ is a rational function if and only if there
are polynomials $p$ and $q$ such that $f$ has the form $p e^q$, the complex exponential functions.

\begin{figure}[ht]
	\centering
	\includegraphics[scale=.31]{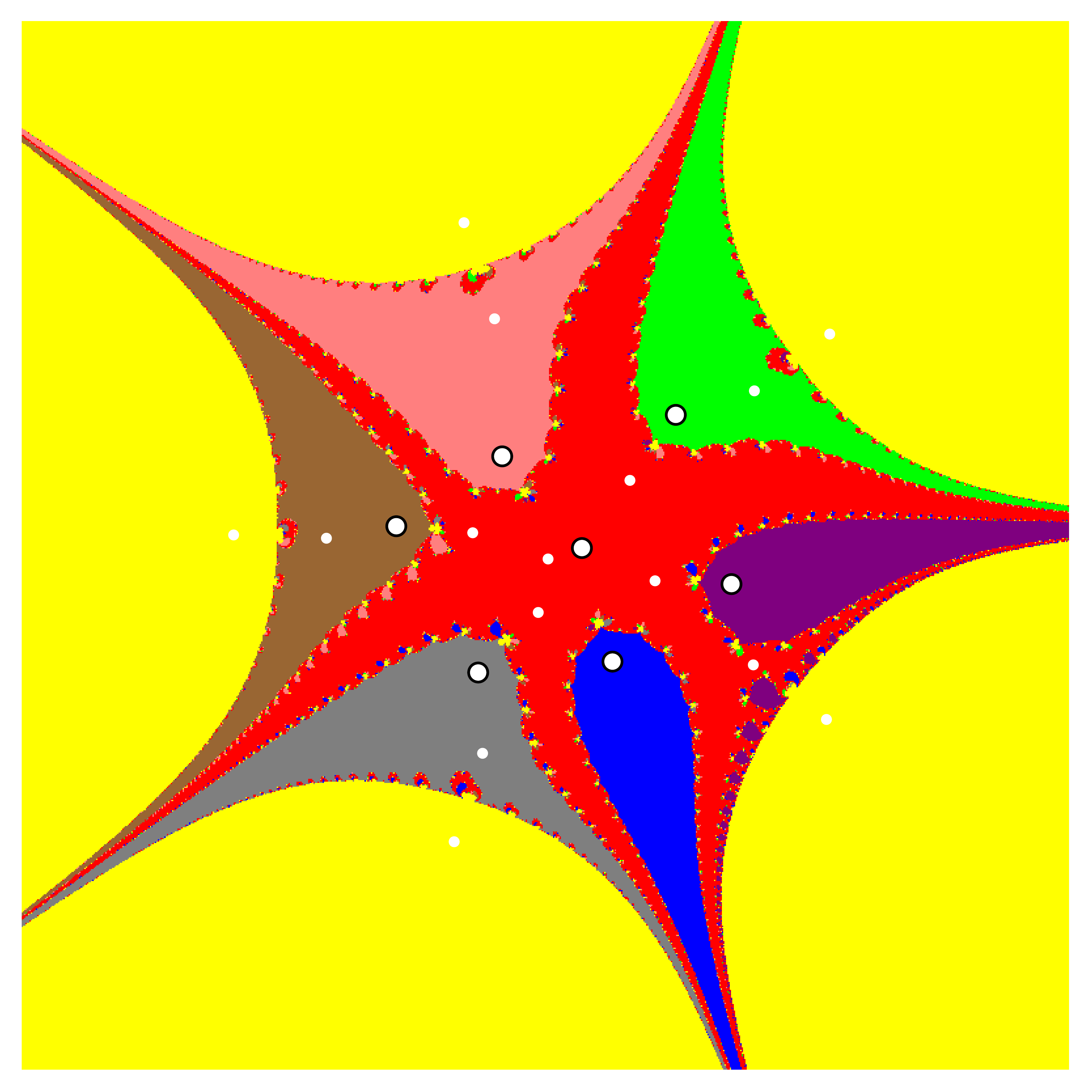}
	\caption[The Julia set of the Newton map of degree $12$.]{The Julia set of the Newton map of degree $12$. Yellow is the basin of parabolic fixed point at $\infty$ with $5$ petals, thick white dots with black circle boundary are fixed points,and white dots are critical points.}
	\label{Fig:NewtonBasin}
\end{figure}

Let $\xi$ be an attracting or a parabolic fixed point with multiplier $+1$ of $R$. The \emph{basin} $\cA(\xi)$ of $\xi$ is \(\text{int}\{z \in \Chat:\lim_{n\rightarrow \infty}R^{\circ n}(z)=\xi\},\) the interior of the set of starting points $z$ which converge to $\xi$ under iteration. An \emph{immediate basin} $\cA^{\circ}(\xi)$ of $\xi$ is the forward invariant connected component of the basin. For a parabolic fixed point there could be more than one immediate basin. Every immediate basin of a parabolic fixed point with multiplier $+1$ contains an open simply connected forward invariant domain where every orbit passing through the immediate basin must visit this domain and stays there. This domain is called an \emph{attracting petal}. The basin of an attracting or a parabolic periodic point is defined similarly. 

The point at $\infty$ for a Newton map $N_{p e^q}$ is parabolic with multiplier $+1$, and there exist $k$ attracting petals at $\infty$, where $k$ is exactly the degree of $q$ (see \cite{Haruta, RS}.) 

\begin{xrem} For a Newton map $N_{p e^q}$ the basin of the parabolic fixed point at $\infty$ can be understood as a virtual basin (see \cite{MS} and \cite{RS} for the definition of a virtual basin for meromorphic Newton maps). 
\end{xrem}
The immediate basin of a fixed point of a rational Newton map is simply connected and unbounded \cite{MS, Przytycki}. M. Shishikura proved that not only \emph{immediate} basins are simply connected but \emph{all} components of the Fatou set are simply connected for every rational function with a single weakly repelling fixed point \cite{Shishikura}. As a corollary of Shishikura's result we obtain that the {\em Julia set} for a rational \emph{Newton map} is \emph{connected}. Recently, K.~Bara\'nski, N.~Fagella, X. Jarque, and B. Karpi\'nska generalized this result to the setting of meromorphic Newton maps proving that the Julia set is connected for all Newton maps of entire functions \cite{BFJK14}. In \cite{BFJK17}, they proved that invariant Fatou components of Newton maps of entire functions, the Newton map is transcendental in general, have accesses to infinity. 

\begin{defin}[Invariant access to $\infty$]
\label{Def:Access}
Let $\cA^{\circ}$ be the immediate basin of a fixed point $\xi \in \C$ or the parabolic fixed point
at $\infty$ of a Newton map $N_{p e^q}$. Fix a point $z_0\in \cA^{\circ}$, and consider a curve $\Gamma:[0,\infty)\to \cA^{\circ}$ with $\Gamma(0) =z_0$ and $\lim_{t\to\infty}\Gamma(t)=\infty$. Its homotopy class (with endpoints fixed) within
$\cA^{\circ}$ defines an {\em invariant access to $\infty$} in $\cA^{\circ}$.
\end{defin}
As an initial point of a curve $\gamma'$ to $\infty$ we may consider any other point $z_0'\in \cA^{\circ}$. We take any curve $\gamma_0$ joining $z_0$ to $z_0'$; then $\gamma_0\cup \gamma'$ is a curve starting at $z_0$ and landing at $\infty$ in the same homotopy class of $\gamma'$, thus the choice of $z_0$ in the definition is not relevant.

Consider a curve $\eta$ that joins $z_0$ to $N_{p e^q}(z_0)$. Then for any curve $\Gamma:[0,\infty)\to \cA^{\circ}$ with $\Gamma(0) =z_0$ and $\lim_{t\to\infty}\Gamma(t)=\infty$, the curve $N_{p e^q}(\Gamma)\cup \eta$ belongs to the invariant access defined by $\Gamma$. This type of invariant accesses to $\infty$ (all invariant access of Newton maps are of this type) is called {\em strongly invariant access} in \cite{BFJK17}.
In the case of immediate basins of the parabolic fixed point at $\infty$, there always exists a special type of an invariant access called {\em dynamical access} to $\infty$. We obtain this access by considering any curve $\eta$ that joins $z_0$ to $N_{p e^q}(z_0)$, and taking the homotopy class of the curve $\Gamma:=\bigcup_{k\ge 0}N_{p e^q}^{\circ k}(\eta)$, which lands at $\infty$ and is forward invariant under $N_{p e^q}$. This dynamical access is of great importance for obtaining a correspondence between the spaces of Newton maps \cite{Ma}.

Every immediate basin of a Newton map has one or several invariant accesses to infinity. The following is a partial case of more general result on accesses within Fatou components of Newton maps of all entire functions \cite[Corollary C]{BFJK17}.
\begin{prop}[Invariant accesses to $\infty$]
\label{Thm:access}
Let $N_{p e^q}$ be a Newton map of degree $d\geq 3$, and $\cA^{\circ}$ be an immediate basin of a fixed point $\xi$ of $N_{p e^q}$ (an attracting or the parabolic fixed point at $\infty$). Assume that $\cA^{\circ}$ contains $k$ critical points of $N_{p e^q}$ (counting multiplicities), then the restriction $N_{p e^q}|_{\cA^{\circ}}$ is a branched covering map of degree $k+1$, and $\cA^{\circ}$ has exactly $k$ distinct invariant accesses to $\infty$. In the case of a parabolic immediate basin of $\infty$, there always exists one dynamical access among $k$ invariant access to $\infty$.  
\end{prop}

For $1\le i \le d$, let $a_i$ be the superattracting fixed points of a postcritically finite Newton map $N_p$ of degree $d$ and let $\cA^{\circ}_i$ be their immediate basins. Let $\phi_i:(\cA^{\circ}_i,a_i) \to (\mathbb{D},0)$ be a Riemann map (global
 B\"ottcher coordinate) with the property that $\phi_i(N_p(z))=\left(\phi_i(z)\right)^{k_i}$ for each $z \in \D$, where $k_i-1\geq 1$ is the multiplicity of $a_i$ as a critical point of $N_p$. The map $z \mapsto z^{k_i}$ fixes the $k_i-1$ internal rays in $\mathbb{D}$. Under $\phi_i^{-1}$ these rays map to the $k_i-1$ pairwise disjoint (except for endpoints) simple curves $\Gamma^1_i,\Gamma^2_i,\ldots,\Gamma^{k_i-1}_{i}\subset \cA^{\circ}_i$ that connect $a_i$ to $\infty$. These curves are pairwise non-homotopic in $\cA^{\circ}_i$ (with homotopies fixing the endpoints) and are invariant under $N_p$ as sets. They represent all invariant accesses to $\infty$ of $\cA^{\circ}_i$.
\begin{defin}[Channel Diagram $\Delta$]\label{Def:UnmarkedChanneldiagram}
 	 The union
 	 \[
 	 \Delta = \bigcup_{i=1}^d \bigcup_{j=1}^{k_{i}-1}
 	 \overline{ \Gamma^{j}_{i} }
 	 \]
 	 is called the \emph{channel diagram.} It forms a connected graph in $\Chat$ that joins finite fixed points to $\infty$.
\end{defin}
 
 It follows from the definition that $N_p(\Delta)= \Delta$ as sets. The channel diagram records the mutual locations (embedding) of the immediate basins of $N_p$. 
 
\begin{defin}[Marked Channel Diagram $\Delta^+_n$]\label{Def:MarkedChanneldiagram}
 	Let $N_p$ be a postcritically finite Newton map with superattracting fixed points $a_1,a_2,\ldots,a_d$. Let $\Delta= \bigcup_{i=1}^d \bigcup_{j=1}^{k_{i}-1}\overline{ \Gamma^{j}_{i} }$ be the channel diagram of $N_p$. For each $1\le i \le d$, let us mark at most one fixed ray $\Gamma^{j^*}_{i}$ in the immediate basin of $a_i$. If a fixed ray in the immediate basin of $a_i$ is marked then the basin of $a_i$ as called a \emph{marked basin}. If $n\le d$ rays are marked, then the marked channel diagram is denoted by $\Delta^+_n$.
\end{defin}
 A basin can be marked or unmarked. The marked channel diagram is a channel diagram $\Delta$ with marking, that is extra information about which fixed rays are selected (marked). Parabolic surgery transforms the marked access of postcritically finite Newton map to a dynamical access for postcritically minimal Newton map.

 Let us prove the following technical lemma on the parabolic Blaschke products. We could not find a reference for its proof, so we provide a full proof of it here. Recall that if $z_0$ is a parabolic fixed point of $f$ with the multiplier $+1$, then in local holomorphic coordinates the map can be written as a parabolic germ $z +a_n (z-z_0)^{n+1} + \cdots$ with complex $a_n\not =0$ and integer $n\ge 1$. The number $n$ refers to the number of distinct directions to converge to the fixed point. In the following lemma, $n=2$, and the parabolic fixed point in this case is said to be {\em double parabolic}.
	 \begin{lem}
	 	\label{Lem:Blaschke}
	 	Let $E:\D\to\D$ be a proper holomorphic function of degree $k\ge2$. Assume that $E$ has a critical point at $z=0$ of (maximal) multiplicity $k-1$ and has a double parabolic fixed point at $z=1$ with the multiplier $+1$. Then $$E(z)=P_k(z)=\frac{z^k+a}{1+a z^k} \text{  where  } a=\frac{k-1}{k+1},$$ and the Julia set of $E$ is the unit circle $\S^1$.
	 \end{lem}
	
	 \begin{proof}
	 	Indeed, $E$ extends to the whole Riemann sphere as a Blaschke product. Denote $E(0)=a$, the critical value, and let $$M_a(z):=\frac{1-\bar{a}}{1-a}\frac{z-a}{1-\bar{a}z}$$ be an automorphism of $\D$ that fixes $z=1$. Then the map $M_a\circ E$ is a proper map of $\D$ such that $M_a( E(0))=0$, and $z=0$ is the only critical point with maximum multiplicity $k-1$. Then $M_a\circ E(z)=u z^k$ for some constant $|u|=1$. Since both $M_a$ and $E$ fix $z=1$, we conclude that $u=1$. Then we obtain 
\[E(z)=\dfrac{z^k+a\frac{1-\bar{a}}{1-a}}{\frac{1-\bar{a}}{1-a}+\bar{a} z^k}.\] 
By assumption, $E(z)$ has a double parabolic fixed point at $z=1$. This means that the fixed point equation $E(z)=z$ has a triple solution at $z=1$. After simplification, the equation becomes \[ z^{k+1}-\dfrac{1}{\bar{a}}z^k+\dfrac{1-\bar{a}}{\bar{a}(1-a)}z-\dfrac{1-\bar{a}}{\bar{a}(1-a)}=0.\]
Taking the second derivative of the left hand side of the latter, we obtain $$(k+1)k z^{k-1}-k(k-1)\dfrac{1}{\bar{a}}z^{k-2}.$$ Requiring that $z=1$ be its solution, we deduce $a=\frac{k-1}{k+1}$. Thus $E(z)=P_k(z)=\dfrac{z^k+a}{1+a z^k}$, where $a=\frac{k-1}{k+1}$. The double parabolic fixed point (two attracting petals with two basins, the open unit disc and the complement of the closed unit disc) condition makes the Julia set connected, the unit circle $\S^1$.
\end{proof}

Postcritically minimal Newton maps of the family \eqref{Eq:peq} are not far from postcritically finite Newton maps of polynomials.
\begin{prop}[Normal forms for postcritically minimal Newton maps]
\label{Thm:Characterization_PCM}
Let $N_{p e^q}$ be a PCM Newton map, $U$ be any connected component of the Fatou set of $N_{p e^q}$, and let $V=N_{p e^q}(U)$. Then $U$ contains a unique center $\xi_U$. Moreover, there exist Riemann maps $\psi_U:U\to \D$ with $\psi_U(\xi_U)=0$ and $\psi_V:V\to \D$ with $\psi_V(\xi_V)=0$ such that
\begin{enumerate}[(a)]
\item if $U$ is an immediate basin of a parabolic fixed point at $\infty$ (in this case $V=U$), then the following diagram is commutative
\[
\begin{diagram}[heads=LaTeX,l>=3em]
U&\rTo^{N_{p e^q}}   & U\\
\dTo^{\psi_U} && \dTo_{\psi_U}\\
\D & \rTo^{P_k}& \D,
\end{diagram}
\]
where $P_k(z)=\frac{z^k+a}{1+a z^k}$ with $a=\frac{k-1}{k+1},$ the parabolic Blaschke product, and $k-1\ge 1$ is the multiplicity of the center of $U$ as a critical point of $N_{p e^q}$;
\item in all other connected component of the Fatou set (also including periodic ones), we have the following commutative diagram
\[
\begin{diagram}[heads=LaTeX,l>=3em]
 U&\rTo^{N_{p e^q}}& V\\
\dTo^{\psi_U}& &\dTo_{\psi_V}\\
\D & \rTo^{z\mapsto z^k}&\D
\end{diagram}
\]
where $k-1$ is the multiplicity of the center of $U$ as a critical point of $N_{p e^q}$, if the center is not a critical point then we let $k=1$.
\end{enumerate}
\end{prop}
Postcritically finite Newton maps of polynomial do not have a parabolic fixed point, thus the first diagram is the key difference between postcritically \emph{finite} Newton maps of polynomials and postcritically \emph{minimal} Newton maps of complex exponential functions.

\begin{proof}[Proof of Proposition \ref{Thm:Characterization_PCM}]
To simplify notation, let $R=N_{p e^q}$. The Julia set of the PCM Newton map $R$ is connected, hence every connected component of the Fatou set is simply connected. It is sufficient to show \emph{uniqueness} of a center in $U$ as there always exists such a center by Definition~\ref{Def:Center}. If $U$ is an \emph{immediate basin} of a superattracting periodic point (fixed points are also counted as periodic points) then the center is the critical periodic point, which is unique. If $U$ is a fully invariant Fatou component (superattracting or parabolic) then we are done. Let $T$ be a component of $R^{-1}(U)$ other than $U$. Let $C_T$ be the union of the set $R^{-1}(\xi_U)\cap T$ and the set of critical points of $R$ in $T$. Then by Definition~\ref{Def:PCM} we have $R(C_T)=\{\xi_U\}$. We want to show that $C_T$ is a single point. This follows from the Riemann-Hurwitz formula since $R:T\setminus C_T\to U\setminus \xi_U$ is unramified covering. Alternatively, observe that $\pi_1(U\setminus \xi_U)=\Z$ and the induced map $R_*:\pi_1(T\setminus C_T)\to \pi_1(U\setminus \xi_U)$ is injective, hence $C_T$ is a single point, denote it by $\xi_T$. In other words, $R^{-1}(\xi_U)\cap T=\{\xi_T\}$. Induction finishes the argument for all other iterated preimages of $U$, thus proving the uniqueness of a center in every component of the Fatou set.

We finish the proof of the theorem by showing the existence of Riemann maps with the suitable commutative diagrams. Let $U$ be a parabolic \emph{immediate basin} of $\infty$ and $\psi_U$ be its Riemann map sending its center $\xi_U$ to the origin. The Riemann map $\psi_U$ is uniquely defined up to post-composition by a rotation about the origin. The point at $\infty$ has many accesses, among which there exists a unique dynamical access. After rotation,  thus normalizing the Riemann map, we assume that this access goes to $z=1$. Denote \(F=\psi_U\circ R \circ \psi^{-1}_U,\) which is a proper holomorphic map, of degree $k$, of the unit disc with a single critical point of multiplicity $k-1$ at the origin. The map $F$ has an extension to the closed unit disk, still denoted by $F$. Then $z=1$ is a double parabolic fixed point for $F$. By Lemma~\ref{Lem:Blaschke}, we obtain $F(z)=P_k(z)=\frac{z^k+a}{1+a z^k}$, where $a=\frac{k-1}{k+1}$.

For all other cases (periodic or not), let $U$ be a connected component of the Fatou set and let $V=R(U)$. Let $\psi_U$ and $\psi_V$ be the Riemann maps of $U$ and $V$ sending their centers $\xi_U$ and $\xi_V$ to the origin respectively (if $U$ is forward invariant, then we let $V=U$ and $\psi_V=\psi_U$). Then $F=\psi_V\circ R \circ \psi^{-1}_U$ is a proper map of the unit disc with the only fixed point at the origin. We have $F(z)=u z^k$ with $|u|=1$, where $k$ is a natural number, the degree of $F$. Replace $\psi_V$ and $\psi_U$ by $\mu \psi_V$ and $s \psi_U$ respectively (we denote them again by $\psi_V$ and $\psi_U$ and the composition map by $F$). Then $F(z)=s\mu^{-k}u z^k$; with the choice of $s\mu^{-k_U}u=1$ we obtain $F(z)=z^{k_U}$, where $k_U-1$ is the multiplicity of the critical point $\xi_U$ ($k_U=1$ if the center $\xi_U$ is not a critical point of $R$). If $U=V$ i.e. $U$ is a superattracting immediate basin, then we have $s^{1-k}u=1$, hence $u=s^{k-1}$, so we have $k-1$ choices for $\psi_U$. If $U$ is not an immediate basin ($U$ could be a parabolic component as well) i.e. $U\not =V$ then for a fixed choice of $\psi_V$ we have $\mu^{-k_U}u=1$, hence $u=\mu^{k_U}$, and there are $k_U$ choices for $\psi_U$.
\end{proof}

\begin{xrem}
It is worth mentioning that for an attracting component of the Fatou set the existence of the Riemann map satisfying the second commutative diagram of the previous theorem is given in \cite[Proposition~4.2]{DH1}, where the proof was carried out for postcritically finite polynomials. But the same proof also works for every attracting component of the Fatou set of postcritically minimal Newton maps, since a rational Newton map restricted to these components behaves like a postcritically finite polynomial does.
\end{xrem}

\section{Background on Parabolic surgery for Newton maps}\label{Sec:Haissinsky}

Turning hyperbolics (attracting and repelling fixed points) into parabolic fixed points or perturbing parabolic fixed points into hyperbolics is a big issue in complex dynamics. In \cite[page~16]{GM}, Goldberg and Milnor formulated the following conjecture: for a polynomial $p$ with a parabolic cycle there exists a small perturbation of $p$ such that the immediate basin of the parabolic cycle of $p$ is converted to the basins of some attracting cycles; and the perturbed polynomial on its Julia set is topologically conjugate to $p$ when restricted to the Julia set. Affirmative answers to the conjecture for the case of geometrically finite rational functions were given by many people, including G. Cui, P. Ha\"issinsky, T. Kawahira, and Tan Lei (see \cite{CT1, CT2, Ha, Kaw1, Kaw2}). We remark that the local dynamics near repelling and parabolic fixed points are never conjugate to each other. Any quasiconformal conjugacy to a repelling germ is again a repelling germ.

Let us recall definitions and properties of quasiconformal and David homeomorphisms. Let $U$ and $V$ be domains in $\C$. Let $K\ge 1$, and set $k:=\frac{K-1}{K+1}$. The orientation preserving homeomorphism $\phi:U\to V$ is called \emph{$K$-quasiconformal} if the partial derivatives $\partial\phi$ and $\bar{\partial}\phi$ exist in the sense of distributions and belong to $L^2_\text{loc}(U)$ (i.e. are locally square integrable) and they satisfy $|\bar{\partial}\phi|\le k |\partial\phi|$ in $L^2_\text{loc}(U)$. Quasiconformal homeomorphisms are used in dynamics to work with functions that are not conformal; leaving the space of holomorphic maps we gain be more flexible but are still able to recover the conformal structure using the Measurable Riemann Mapping Theorem and Weyl's lemma. The following properties of quasiconformal homeomorphisms are of great importance for our purpose.

\begin{itemize}
\item If $\phi$ is a $K$-quasiconformal homeomorphism then so is its inverse;
\item Quasiconformal removability of quasiarcs: If $\Gamma$ is a {\em quasiarc} (the image of a straight line under a quasiconformal homeomorphism) and $\phi:U\to V$ a homeomorphism that is $K$-quasiconformal on $U\setminus\Gamma$, then $\phi$ is $K$-quasiconformal on $U$, and hence $\Gamma$ is quasiconformally removable. In particular, points, lines, and smooth arcs are quasiconformally removable.
\item Weyl's Lemma: If $\phi$ is $1$-quasiconformal, then $\phi$ is conformal. In other words, if $\phi$ is a quasiconformal homeomorphism and $\partial_{\bar{z}} \phi=0$ almost everywhere, then $\phi$ is conformal.
\end{itemize}

In the following, an extension of quasiconformal maps by relaxing the uniform upper bound on dilatation is given. It is used by Ha\"issinsky for his parabolic surgery.
Let us first define the David--Beltrami differentials.

\begin{defin}[David--Beltrami differential]
A measurable Beltrami differential
$$
\mu=\mu(z)\frac{d\bar{z}}{dz} \text{ on a domain } U\subset\C
$$
with $||\mu||_\infty=1$ is called \emph{David--Beltrami differential} if there exist constants $M>0$, $\alpha>0$, and $\epsilon_0>0$ such that

\begin{equation}\label{Eq:Beltrami}
\text{Area}\{z\in U:|\mu_\phi(z)|>1-\epsilon\}<Me^{-\frac{\alpha}{\epsilon}}\ \text{for}\ \epsilon<\epsilon_0.
\end{equation}

\end{defin}

\begin{defin}[David homeomorphism]
An orientation preserving homeomorphism $\phi:U\to V$ for domains $U$ and $V$ in $\Chat$ is called David homeomorphism (David map or David) if it belongs to the Sobolev class $W_\text{loc}^{1,1}(U)$, i.e. has locally integrable distributional partial derivatives in $U$, and its induced David--Beltrami differential
\[
\mu_\phi:=\frac{\partial_{\bar{z}} \phi}{\partial_z \phi}\frac{d\bar{z}}{dz}
\]
satisfies the inequality~\eqref{Eq:Beltrami}.
\end{defin}

The condition in the definition is referred to as the \emph{area condition.} The area in the definition is the spherical area. For domains that are bounded, we use the Euclidean area instead. If we let $K_\phi(z):=\frac{1+|\mu_\phi(z)|}{1-|\mu_\phi(z)|}$, the real dilatation of $\phi$, then, equivalently, we can state the area condition as the following: there exist constants $M>0$, $\alpha>0$, and $K_0>1$ such that
\[
\text{Area}\{z\in U:K_\phi(z)>K\}<Me^{-\alpha K}\ \text{for}\ K>K_0.
\]

David maps are different from quasiconformal homeomorphisms in many respects; for instance, the inverse of David map may not be David. There is no Weyl's type lemma for David maps, see \cite{Zak} for more differences between quasiconformal and David maps.

But, similarly to quasiconformal homeomorphisms, quasiarcs are removable for David maps, so we can glue two David maps along a nice boundary to obtain a global David map.

\begin{thm} \cite[David Integrability Theorem]{BF, David}
\label{Thm:David_Integration}
Let $\mu$ be a David--Beltrami differential on a domain $U\subset\Chat$. Then there exists a David homeomorphism $\phi:U\to V$ whose complex dilatation $\mu_\phi$ coincides with $\mu$ almost everywhere.

The integrating map is unique up to post-composition by a conformal map, i.e. if $\widetilde \phi:U\to \tilde V$ is another David homeomorphism such that $\mu_{\widetilde\phi}=\mu$ almost everywhere, then $\widetilde \phi \circ \phi^{-1}:V\to\widetilde V$ is conformal.
\end{thm}

The following is parabolic surgery for polynomials. 
\begin{thm}[P.~Ha\"issinsky] \cite{Ha} \label{Thm:Haissinsky_Surgery}
Let $p$ be a polynomial of degree $d\ge 2$ with a (non-super)attracting fixed point $\alpha$ with basin of attraction $\cA(\alpha)$ and a repelling fixed point $\beta\in\partial \cA^{\circ}(\alpha)$, and assume $\beta\not \in P_p$. Suppose also that $\beta$ is accessible from the basin $\cA^{\circ}(\alpha)$. Then there exist a polynomial $q$ of degree $d$ and a David homeomorphism $\phi:\Chat\to\Chat$ such that
\begin{enumerate}[(a)]
\item for all $z\not \in \cA(\alpha)$, we have $\phi\circ p(z)=q\circ\phi(z)$; in particular, $\phi:J(p)\to J(q)$ is a homeomorphism which conjugates $p$ to $q$ on the Julia sets;
\item $\phi(K(p))=K(q)$, $\phi(\beta)$ is a parabolic fixed point of $q$ with the multiplier $q'(\phi(\beta))=1$, and
$\phi(\cA^{\circ}(\alpha))=\cA^{\circ}(\phi(\beta))$;
\item $\phi$ is conformal in the interior of $\Chat\setminus \cA(\alpha)$.
\end{enumerate}
\end{thm}

\begin{xrem} Note that the construction of the above surgery procedure is local (along the given access) and can be performed on rational functions that are not necessarily polynomials. By a slight change of the proof, this theorem can be stated replacing the attracting fixed point by several cycles and the repelling point by cycles such that their periods divide those of the attracting points related to them, providing the repelling points that are to become parabolic are not accumulated by the recurrent critical orbits of $p$. In particular, it is possible to apply surgery when some critical point eventually lands at the repelling fixed point that is going to become parabolic. We shall go through the whole process of the proof following \cite[Section 9.3]{BF} but providing more details.
\end{xrem}

We shall use the following classical result on lifting property of covering maps \cite[Section~1.3]{Hatch}.
\begin{lem}
\label{Lem:Lifting}
Let $Y,Z$, and $W$ be path-connected and locally path-connected Hausdorff spaces with base points $y\in Y, z\in Z$, and $w\in W$. Suppose $p:W\to Y$ is an unbranched covering and $f:Z\to Y$ is a continuous map such that $f(z)=y=p(w)$:
\[\begin{diagram}[heads=LaTeX,l>=3em]
Z,z &\rTo^{\tilde f} & W,w\\
& \rdTo_{f} & \dTo_{p}\\
& & Y,y
\end{diagram}\]
There exists a continuous lift $\tilde f$ of $f$ to $p$ with $\tilde f (z)=w$ for which the above diagram is commutative i.e. $f=p\circ \tilde f$ if and only if \[ f_*(\pi_1(Z,z)) \subset p_* (\pi_1(W,w)), \] where $\pi_1$ denotes the fundamental group. This lift is unique if it exists.
\end{lem}

We must mention that parabolic surgery is not directly applicable to superattracting domains. In order to overcome this difficulty, we need to change these domains to (non-super)attracting ones. 	

\begin{lem}\label{Lem:GettingPCM}
Let $N_p$ be a postcritically finite Newton map of degree $d\ge3$ with a superattracting fixed point $\xi$. Let its basin $\cA(\xi)$ with a marked invariant access $\Delta^+_1$ to $\infty$ in the immediate basin $\cA^{\circ}(\xi)$ be given. Then there exist a Newton map $N_{\hat p}$ of degree $d$ and a quasiconformal homeomorphism $\phi:\Chat\to\Chat$ such that
\begin{enumerate}[(a)]
\item $\phi\circ N_p=N_{\hat p}\circ\phi$ except on a compact set in $\cA^{\circ}(\xi)$, in particular, $\phi(F(N_p))=F(N_{\hat p})$ and $\phi(J(N_p))=J(N_{\hat p})$;
\item $\phi(\xi)$ is an attracting fixed point of $N_{\hat p}$ with the multiplier $N^\prime_{\hat p}(\phi(\xi))=\frac12$, and its immediate basin $\cA^\circ(\phi(\xi))=\phi(\cA^\circ(\xi))$ contains a single critical point;
\item $\phi$ is conformal except on $\cA(\xi)$;
\item $\phi(\Delta^+_1)$ is a marked invariant access in $\cA^\circ(\phi(\xi))$;
\item critical orbits in the Julia set and in superattracting basins are finite;
\item all critical points in basin of $\phi(\xi)$ are in \emph{minimal critical orbit relations}: if $c$ is a critical point in a strictly preperiodic component $U$, of preperiod $m\ge 1$, of the basin of $\phi(\xi)$, then $N_{\hat p}^{\circ m}(c)$ is a critical point in $\cA^\circ(\phi(\xi))$.
\end{enumerate}
\end{lem}

\begin{xrem}
	The lemma is still valid if we consider several basins instead of one, we just need to work one by one with all marked basins. We change the multiplier to be $\frac12$ with the goal to obtain the Newton map of some polynomial $\hat p$. But this choice is purely artificial; in fact, it is possible to change the multiplier to be any non-zero $\lambda \in \D^*$. The lemma assures that not only is the multiplier changed at a superattracting fixed point, but we can also keep the marked invariant access unchanged.
\end{xrem}

\begin{proof}[Proof of Lemma~\ref{Lem:GettingPCM}]
The proof we supply is similar to that of the straightening of polynomials-like mappings (see \cite[Chapter 1,
Theorem 1]{DH}).
Let $N_p$ be a postcritically finite Newton map of degree $d\ge3$. Let $\xi$ be its superattracting fixed point with its basin $\cA(\xi)$. Fix a single marked invariant access $\Delta^+_1$ in the immediate basin $\cA^{\circ}(\xi)$. By Proposition~\ref{Thm:Characterization_PCM} the point $\xi$ is the only critical point in $\cA^\circ(\xi)$. Its local dynamics is conjugate to $z\mapsto z^{d_{\xi}}$, where $d_{\xi}>1$ is the multiplicity of $\xi$ as a critical point of $N_p$. Let $\psi:\cA^\circ(\xi)\to\D$ be a Riemann map (B\"ottcher coordinate) such that $\psi(\xi)=0$. It is unique up to rotation; normalize it so that the fixed ray corresponding to the marked invariant access $\Delta^+_1$ maps to $\left[0,1\right]$, which we call the ``zero ray". We want to change the multiplier at the fixed point to $\frac12$ so that we obtain a Newton map of some other polynomial $\hat p$. Then the polynomial $\hat p$ should have a single double root (other roots are simple) and thus its degree is $d+1$.
	 	
For a given $ 0\le b<\frac{k-1}{k+1}$, the Blaschke product $B_b(z)=\frac{z^k+b}{1+b z^k}$, considered as a proper self map of $\D$, has a unique attracting fixed point $\alpha \in [0,1)$, which attracts every point in $\D$. Note that $b<\alpha$, otherwise if $b\ge \alpha$ then $\alpha^k+b\ge\alpha^k +\alpha$ this yields $\alpha(1+b \alpha^k)\ge\alpha^k +\alpha$ since $\alpha$ is a fixed point. After cancellations, we obtain $b\alpha\ge1$, a contradiction. The multiplier at $\alpha$ depends continuously on $b$. When $b=0$ the Blaschke product is $z\mapsto B_0(z)=z^k$, then $B_0'(0)=0$. If $b\nearrow \frac{k-1}{k+1}$ then the multiplier at $\alpha$ converges to 1; in particular, $\alpha$ also converges to 1. (It can be shown that the multiplier map is an increasing function of $b$ in the interval $[0,\frac{k-1}{k+1}]$, but this is not relevant for us.) We choose $b=b_0$ such that the multiplier at the fixed point $\alpha=\alpha_0$ is $1/2$, i.e $B_{b_0}'(\alpha_0)=1/2$. To ease notations, let us drop the index $b_0$ in $B_{b_0}$ and simply call it $B$.

Fix $\alpha_0<r<1$. Set $\D_r=\{|z|<r\}$ and $\S_r^1=\{|z|=r\}$. Then $B_0^{-1}(\D_r)=\D_{r^{1/k}}$ and $B_0^{-1}(\S_r^1)=\S^1_{r^{1/k}}$. Denote $A_0=\overline{\D_r}\setminus \D_r^k$ and $A_1=\overline{\D_r}\setminus B(\D_r)$ closed annuli. Note that $B(\D_r)$ is a disk and $M_b(\D_r^k)=B(\D_r)$, for the M\"obius map $M_b(z)=\frac{z+b}{1+bz}$ with $b=b_0$. Moreover, $M_b$ is biholomorphism (the Riemann map) from $\D_r^k$ to $B(\D_r)$. Consider $M_b:S^1_{r^k}\to M_b(S^1_{r^k})$, and lift it to the map $\psi:S^1_r\to S^1_r$ such that $M_b\circ z^k=B\circ\psi$ and fixing $r$. Then $\psi=\text{id}$. Using \cite[Proposition~2.30]{BF}, interpolate $\text{id}:\S_r^1\to\S_r^1$ and $M_b:S^1_{r^k}\to M_b(S^1_{r^k})$ to obtain a quasiconformal map $h$ between $A_0$ and $A_1$. Define a quasiregular function (a composition of holomorphic and quasiconformal map) $g:\D\to\D$ as follows
 \[g(z) = \left \{
 \begin{array}{ll}
 	\mbox{$z^k$,} & \mbox{$z\in \D\setminus \D_r$,}\\
 	\mbox{$M^{-1}_b\circ B\circ h(z)$,} & \mbox{$z\in A_0$,}\\
 	\mbox{$M^{-1}_b B\circ M_b(z)$,} & \mbox{$z\in \overline{\D}_{r^k}.$}
 \end{array} \right. \]

\begin{figure}[h!]
	\centering
	\includegraphics[scale=.5]{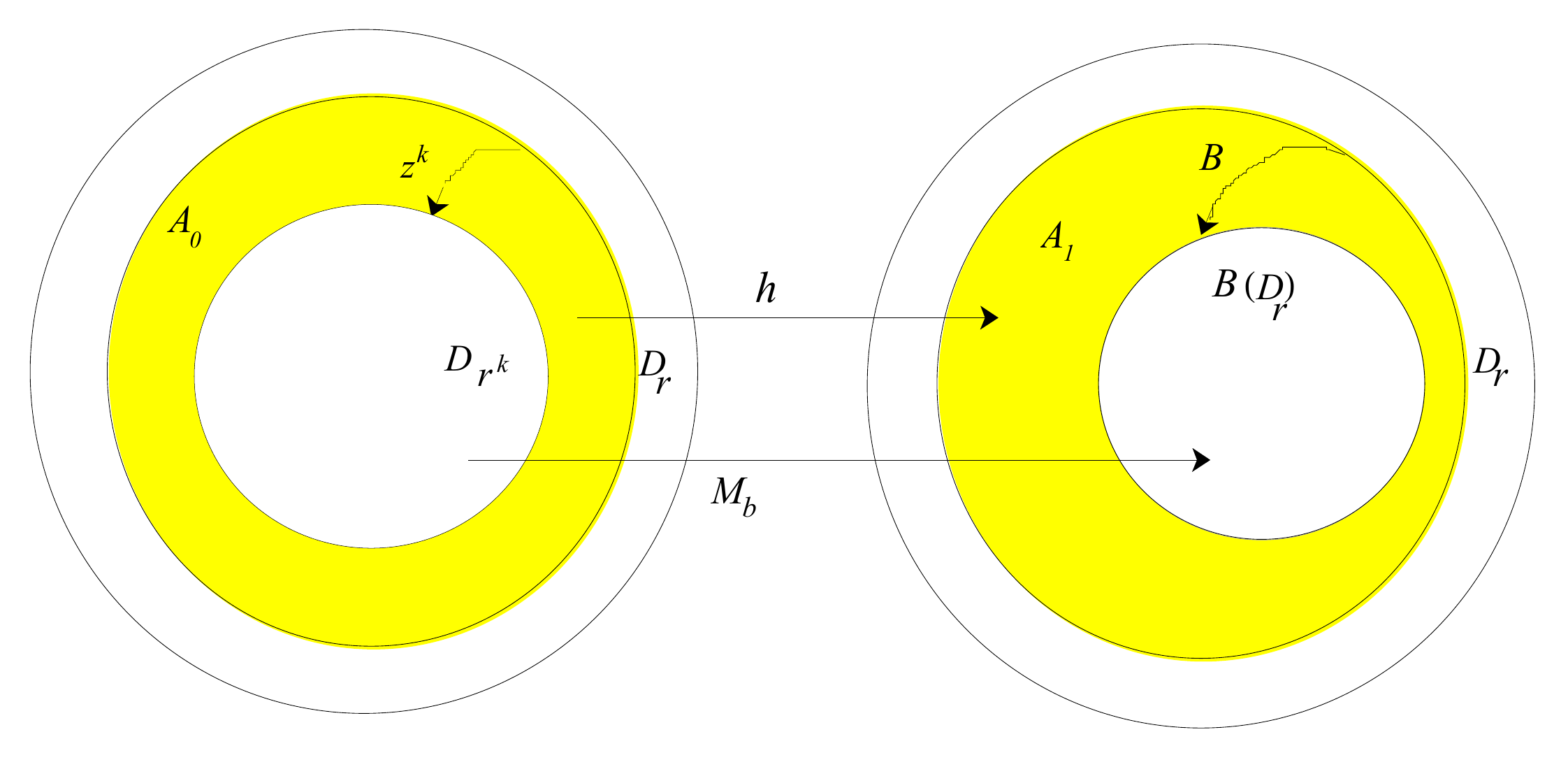}
	\caption{Surgery on the unit disk. Two outer disks are the unit disks, the two annuli in yellow are $A_0$ and $A_1$ and the inner white disks are $\D_{r^k}$ and $B(\D_r)$ respectively.}
	\label{Fig:surgery}
\end{figure}

If we simplify the above formula for $g$, we obtain
 \[g(z) = \left \{
 \begin{array}{ll}
 	\mbox{$z^k,$} & \mbox{$z\in \D\setminus \D_r$,}\\
 	\mbox{$h^k(z),$} & \mbox{$z\in A_0$,}\\
 	\mbox{$M_b^k(z),$} & \mbox{$z\in \overline{\D}_{r^k}.$}
 \end{array} \right. \]
Fig.~\ref{Fig:surgery} shows how we cut and past new dynamics in the unit disk.

 Note that $g(M_{b_0}^{-1}(\alpha_0))=M_{b_0}^{-1}(\alpha_0)$ and $z=-b_0$ is a critical point of maximal multiplicity $k-1$. Also note that the interval $\left[r^{1/k},1\right]$ is invariant. Since $z=1$ is a repelling fixed point for $B_0$, the interval $\left[r^{1/k},1\right]$ is contained in the marked invariant access coming from $B_0$ and is invariant for $g$ by the construction. We shall construct $\mu$, a $g$-invariant Beltrami form. Let $\mu_0$ be the conformal structure, define $\mu$ on $\D_{r^k}$ as $\mu=\mu_0$. In the annulus $A_0$ define it as a pullback by $h^k$, i.e. $\mu=(h^k)^*\mu_0$. Finally, we extend $\mu$ to the rest of $\D$ by the dynamics of $g$, observing that for every $z\in\D\setminus \D_r$, there is a unique $n \geq 0$ such that $g^{\circ n}(z)$ belongs to the half open annulus $A_0\setminus \S_{r^k}^1$; moreover, $g^{\circ n}(z)=B_0^{\circ n}(z)$. Hence $\mu$ is defined recursively as
 \[\mu = \left \{
 \begin{array}{ll}
 \mbox{$\mu_0$} & \mbox{on $ \D_r$},\\
 \mbox{${(h^k)}^*\mu_0$} & \mbox{on $A_0$}, \\
 \mbox{$(B^n_0)^*\mu$} & \mbox{on $B^n_0(A_0).$}
 \end{array} \right. \]
 By our construction, $\mu$ is $g$-invariant and its maximum dilatation satisfies $||\mu||=||((h)^k)^*\mu_0||<1$. 

Recall that $\psi:\cA^\circ(\xi)\to\D$ is the Riemann map. 
 Let us transport $g$ to the immediate basin $\cA^{\circ}(\xi)$ of $N_p$ and define a continuous function  \[F(z) = \left \{
 \begin{array}{ll}
 \mbox{$\psi^{-1}\circ g\circ\psi(z)$,} & \mbox{$z\in\cA^{\circ}(\xi)$},\\
 \mbox{$N_p(z),$} & \mbox{$z\not\in\cA^{\circ}(\xi).$}
 \end{array} \right. \]
 Note that $F(z)=N_p(z)$ for $z\in \psi^{-1}(\D\setminus \D_r)$.
 We will now define an $F$-invariant Beltrami form $\mu_F$ in $\Chat$. In the immediate basin $\cA^\circ(\xi)$ define $\mu_F$ as the pull back of $\mu$ by the Riemann map $\psi:\cA^\circ(\xi)\to\D$; for other components of the basin $\cA(\xi)$ we spread it by the dynamics of $F$. We put the standard complex structure on the complement of $\cA(\xi)$, thus obtaining an $F$-invariant Beltrami form $\mu_F$. We apply the Measurable Riemann Mapping Theorem \cite[Theorem~1.28]{BF} to $\mu_F$ deducing a quasiconformal homeomorphism $\phi:\Chat\to\Chat$, unique up to an automorphism of $\Chat$. The conjugation $\phi\circ F\circ \phi^{-1}$ is a rational function on $\Chat$. Let us normalize $\phi$ as to fix $\infty$. Then $\phi\circ F\circ \phi^{-1}$ is a Newton map of some polynomial, denote the polynomial by $\hat p$. Thus $N_{\hat p}=\phi\circ F\circ \phi^{-1}$ is the resulting function. The maps $N_p$ and $N_{\hat p}$ are conjugate in some neighborhoods of their Julia sets. The conjugating map transforms invariant accesses in $\cA^{\circ}(\xi)$ of $N_p$ to that of $N_{\hat p}$, so we obtain a marked invariant access of $N_{\hat p}$ corresponding to the marking $\Delta^+_1$. We still use the same notation $\Delta^+_1$ for the resulting access and channel diagram, even though the channel diagram is defined only for postcritically finite Newton maps. We have marked invariant accesses but the new marked channel diagram is not invariant under the new function, but it is invariant near $\infty$ this is what all we need for our application of parabolic surgery. By construction, all the conditions of the lemma are satisfied.
	 \end{proof}

\begin{rem}\label{Rem:Blaschke}
In the proof of the previous lemma, let us straighten the $g$ invariant conformal structure on $\D$, normalize it to get $\phi$ that fixes the origin. Set $E=\phi\circ g \circ \phi^{-1}$; then $E$ is a Blaschke product with an attracting fixed point $\alpha$ in $\D$, and the origin is its critical point of maximal multiplicity. Compose $\phi$ with a rotation about the origin so that the critical value of $E$ is real and positive. Since $\phi$ is conformal near $\alpha$, we have $E'(\alpha)=1/2$. It is easy to see that $E(z)=\frac{u z^k+b}{1+b u z^k}$, where $b>0$ is the critical value, and $|u|=1$. Then $1/\bar{\alpha}$ is also a fixed point of $E$ with the multiplier $\overline{E'(\alpha)}$. We obtain a system of equations, of which $\alpha$ and $1/\bar{\alpha}$ are the only solutions:
\begin{align}
\label{Eq:Bla-Att1}
\frac{u z^k+b}{1+b u z^k}&=z\\
\label{Eq:Bla-Att2}
\frac{k u z^{k-1}(1-b^2)}{(1+b u z^k)^2}&=\dfrac12
\end{align}

Solving \eqref{Eq:Bla-Att1} for $u$ and plugging this into \eqref{Eq:Bla-Att2}, we obtain the following quadratic equation \[k(z-b)(1-z b)=\dfrac12 z(1-b^2).\]
By assumption, the last equation should have solutions $\alpha$ and $1/\bar{\alpha}$. This implies that $\alpha>0$, in particular, it is real. Then $u$ is also real and positive, which results to $u=1$. Thus $E(z)=\frac{z^k+b}{1+b z^k}$. By assumption $z=1$ is a repelling fixed point with the real multiplier, which implies that $0<b<\frac{k-1}{k+1}$.
\end{rem}
 
By application of Lemma~\ref{Lem:GettingPCM} to $N_p$ with its marked basins, we obtain a Newton map $N_{\hat p}$ and a quasiconformal map $\phi_0$ such that $\phi_0\circ N_p=N_{\hat p}\circ\phi_0$ except on a compact set.  Moreover, every marked immediate basin of $N_{\hat p}$ is now (non-super)attracting with a single critical point in it, and $N_{\hat p}$ is postcritically minimal on these marked basins, meaning that its critical orbits are in minimal critical orbit relations. 

The compositions $\phi\circ\psi$ and $\psi\circ\phi$ of a quasiconformal homeomorphism $\psi$ with a David homeomorphism $\phi$ are again David homeomorphisms, since a quasiconformal homeomorphism distorts area in a bounded fashion by Astala's theorem \cite{A}. It is enough to work with $N_{\hat p}$, but to simplify notations we still denote it by $N_p$. The relation is $\phi_0\circ N_p=N_{\hat p}\circ\phi_0$ except on a compact set contained in the union of marked immediate basins of $N_p$. 

It is easy to carry out the surgery for the Blaschke products on the unit disk and then transfer it to the marked immediate basins via Riemann maps. Let \[B(z)=\frac{z^k+b}{1+b z^k}\] be the Blaschke product of degree $k\ge2$ for $b=b_0$ as defined in the proof of Lemma~\ref{Lem:GettingPCM} (see also Remark~\ref{Rem:Blaschke}). Considered as a self map of the unit disk, $B(z)$ has a unique critical point of maximal multiplicity $k-1$ at the origin, and its Julia set is the unit circle. Moreover, $B(z)$ has an attracting fixed point $\alpha_0 \in \left(0,1\right)$ with $B'(\alpha_0)=\frac12$, which attracts every point in $\D$, while $z=1$ is a repelling fixed point with the real multiplier $\lambda =\frac{1-b}{1+b}k$. Recall that $P_k(z)$ has a double parabolic fixed point at $z=1$, which also attracts every point in $\D$. In both cases the interval $\left[0,1\right]$ is invariant, and the critical orbit marches monotonically along this segment towards $z=\alpha$ in the case of $B(z)$ and towards $z=1$ in the case of $P_k$. The interval $\left[0,1\right]$ corresponds to the dynamical access for $P_k(z)$.

Now we consider the local dynamics around the repelling fixed point. Let $f(z)=\lambda z$ with $\lambda>1$ be the repelling model map. We define the sector
	\[ S:=\left\{z\in\C\ |\ \theta\le \arg z\le 2\pi-\theta\ \text{and}\ 0<|z|<1/\lambda^m \right\},\]
for $m>m_0$, where $m_0$ is large, and $0<\theta<\pi$ (see Fig.~\ref{Fig:Quadrilateral}). We write $Q_m^f$ for the quadrilaterals bounded by the segments $[(1/\lambda^{m+1})e^{\pm i\theta},(1/\lambda^{m})e^{\pm i\theta}]$ and arcs of radii $1/\lambda^{m+1}$ and $1/\lambda^{m}$ contained in $S$ (see Fig.~\ref{Fig:Quadrilateral}). Note that $Q_m^f=f^{-\circ m}(Q_0^f)$. It is easy to observe that the map $z\mapsto \omega(z)=\frac{Log\lambda}{Log z}$ conjugates $f$ on $\D_{\lambda^{-m}}\setminus S$ to the parabolic model map
	\[g:\omega\mapsto \frac{\omega}{\omega+1}\]
	 on the cusp $C=g(\D_{\lambda^{-m}}\setminus S)$ with vertex at the origin, i.e. $\omega\circ f=g\circ\omega$ on $\D_{\lambda^{-m}}\setminus S$ (see Fig. \ref{Fig:Quadrilateral}).
	 
	\begin{figure}[htbp]
	 		\centering
	 	\includegraphics[scale=.28]{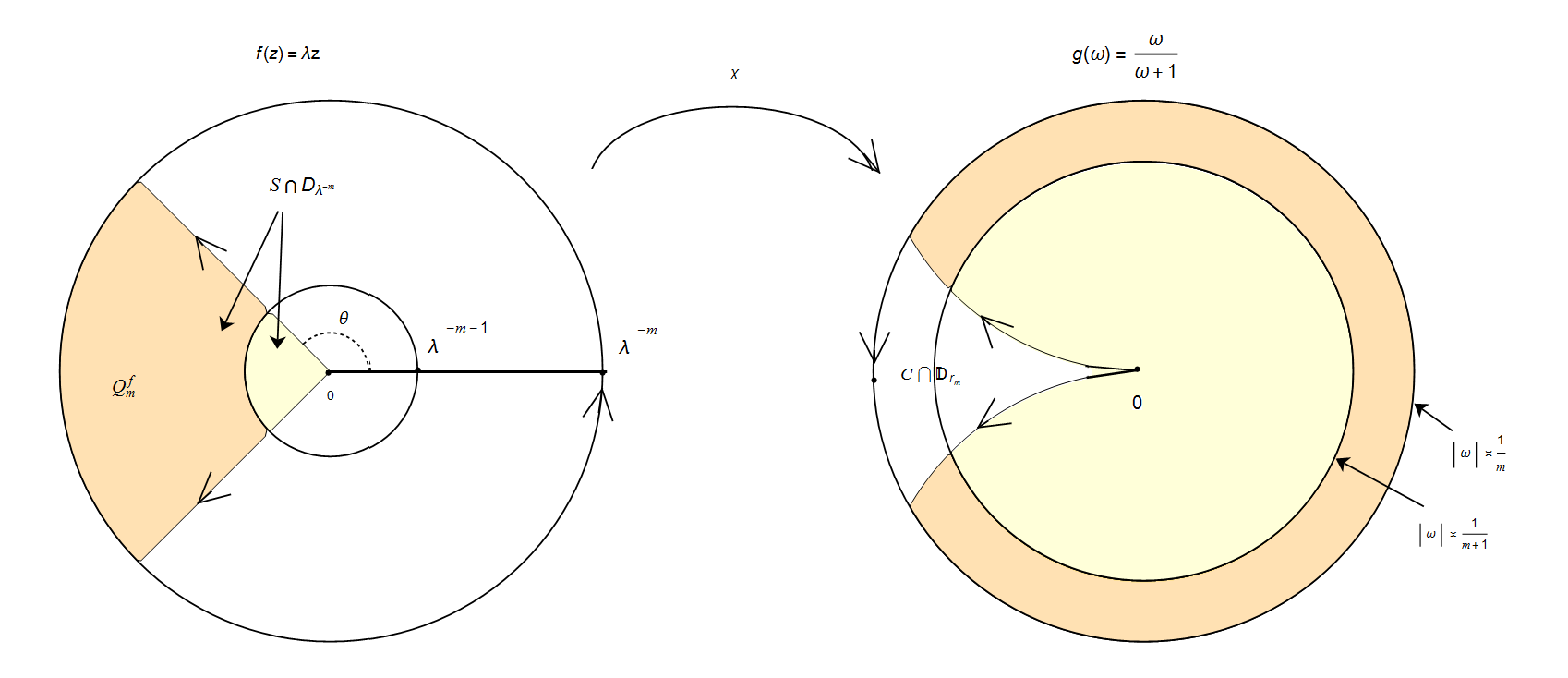}
	 	\caption[Illustration of the function $\chi$.]{Illustration of the function $\chi$. The shaded region on the left	corresponds to $S\cap\mathbf{D}_{\lambda^{-m}}$ and $r_m=|\omega(\lambda^{-m}e^{\pm i \theta})|\asymp\frac{1}{m}$. The initial map $z\mapsto\omega(z)$ sends the white region (the complement of $S$) on the left to the white region on the right.}
	 	\label{Fig:Quadrilateral}
	 \end{figure}
	
	We write
\[
a(t)\asymp b(t)
\]
for two positive valued functions $a(t)$ and $b(t)$ if there exists a constant $C>0$ such that $C^{-1}b(t)\le a(t) \le C b(t)$. 

	\begin{lem}
		\label{Lem:Quadrilateral}
		There exists a David extension $\chi$ of $z\mapsto \omega(z)$ to a neighborhood of the origin with $K_\chi\asymp m$ on $Q_m^f$.
	\end{lem}

	\begin{lem}
		\label{Lem:Parabolic_Surgery_Disk}
		There exist a piecewise $C^1$ homeomorphism $\phi:\D\to\D$ and a sector $S_B\subset\D$ with vertex at $z=1$, which is a neighborhood of $\alpha$, such that
		\begin{enumerate}[(a)]
			\item  for all $z\in \D\setminus S_b$, $\phi\circ B(z)=P_k\circ\phi(z)$;
			\item there is a set $S'_B$, which is the intersection of $S_B$ with a neighborhood of $z=1$, such that $\phi:\D\setminus \bigcup_k B^{-\circ k}(S'_B)\to \phi(\D\setminus \bigcup_k B^{-\circ k}(S'_B))$ is a quasiconformal homeomorphism;
			\item on the quadrilaterals $Q_m^B$ in $S'_B$ defined as $Q_m^f$ for Lemma \ref{Lem:Quadrilateral}, we have $K_\phi\asymp m$ for all $m\ge m_0$.			
		\end{enumerate}
	\end{lem}

	For the proofs of the last two lemmas we refer to \cite[Lemmas 9.20 and 9.21]{BF}.

	 \textbf{Topological surgery.} In order to ease the notations, let us do the construction for one basin only, and denote by $\xi$ an attracting fixed point and by $\cA^\circ$ its immediate basin with a single critical point $c_\xi$ of maximal multiplicity $k-1=\deg(N_p,c_\xi)-1$. Consider the Riemann map \[\psi=\psi_{\cA^\circ}:\cA^\circ\to\D\] such that $\psi(c_\xi)=0$; now we can still post-compose it with a rotation. Let \[E=\psi\circ N_p\circ \psi^{-1}:\D\to\D.\] By construction $E$ has a critical point of maximal multiplicity at the origin and has an attracting fixed point in $\D$ with the multiplier $1/2$. By Remark~\ref{Rem:Blaschke} we obtain $E(z)=\frac{z^k+b}{1+b z^k}$ for some $0<b<\frac{k-1}{k+1}$. Denote $B(z)=E(z)$, then $\psi(\xi)$ is a fixed point of $B$. 

Consider the map $\phi:\D\to\D$ given by Lemma~\ref{Lem:Parabolic_Surgery_Disk}, which is a partial conjugacy between $B$ and $P_k$. Let $\hat{B}= \phi^{-1}\circ P_k\circ \phi$ be the conjugation. Hence, $\hat{B}=B$ except on the sector $S_B$. We transfer this data to the immediate basin: define $F=\psi^{-1}\circ \hat{B}\circ \psi$ on $\cA^\circ$ which coincides with $N_p$ except on the sector $\psi^{-1}(S_B)$.
	
	 Finally, set
	 \[G(z) = \left \{
	 \begin{array}{ll}
	 F(z), & z \in \cA^\circ,\\
	 N_p(z), & z\not\in \cA^\circ.
	 \end{array} \right. \]
	 This map $G$ is our topological model: a ramified covering of degree $d$ and piecewise $C^1$.
	 We have to ensure that all critical points of $G$ satisfy minimal critical orbit relations; we know that $N_p$ does by construction. Let $U$ be a component of the basin $\cA$, of preperiod $m_U\ge1$ i.e. $G^{\circ m_U}(U)=N^{\circ m_U}_p(U)=\cA^\circ$, and let $c_U\in U$ be a critical point in $U$, the center of $U$. Recall that by construction $G^{\circ m_U}(c_U)=N^{\circ m_U}_p(c_U)=c_\xi$, where $c_\xi$ is a critical point of $N_p$ in $\cA^\circ$, its center. By the definition of the Riemann map $\psi=\psi_{\cA^\circ}:\cA^\circ\to\D$ with $\psi(\xi)=\alpha$ and $\psi(c_\xi)=0$, we have 
\[F=\psi^{-1}\circ \hat{B}\circ \psi=\psi^{-1}\circ\phi^{-1}\circ P_k\circ \phi\circ \psi.\]
 It is easy to observe that $B^{-1}(b)=P_k^{-1}(a)=\{0\}$. From $\phi(B(0))=\phi(b)=a=P_k(0)$, it follows that $\phi(0)=0$. We have $F(c_\xi)=\psi^{-1}(b)$ and $F^{-1}(\psi^{-1}(b))=\{c_\xi\}$; hence, $c_\xi$ is a critical point of $F$. This is a crucial property of our model map: after applying the straightening theorem, we obtain a postcritically minimal Newton map.
	
\textbf{Straightening of almost complex structure.} We will define a $G$-invariant almost complex structure $\mu$ in $\Chat$. Let $\hat{\mu}=\partial_{\bar{z}}\phi/\partial_{z}\phi$, which is a David--Beltrami form. It is defined on $\D$ and invariant under $\hat{B}$. We transport it to the immediate basin $\cA^\circ$ by defining the pull back $\mu=\psi^*\hat{\mu}$.
	 We have the commutative diagram
	
	 \[\begin{diagram}[heads=LaTeX,l>=3em]
	 (D,\mu_0)        &\rTo^{B_{par}}   & (D, \mu_0)\\
	\uTo^{\phi} &         & \uTo_{\phi}\\
	 (D, \hat{\mu})        &\rTo^{\hat{B}}   & (D, \hat{\mu})\\
	\uTo^{\psi}  &        & \uTo_{\psi}\\
	(\cA^\circ, \mu)   &\rTo^{F}   & (\cA^\circ, \mu)
	 \end{diagram}\]

We extend $\mu$ recursively by dynamics of $F$ (which is equal to $N_p$ outside of the sector $\psi^{-1}(S_B)$) to the rest of the dynamical plane as

\[
\mu=\left\{
\begin{array}{ll}
(N_p^{\circ k})^*\mu & \text{on} \ N_p^{-\circ k}(\cA^\circ)\ \text{for all $k\ge1$},\\
\mu_0 & \text{on} \ C\setminus \bigcup_{k=1}^\infty N_p^{-\circ k}(\cA^\circ).
\end{array} \right.
\]
By definition the map $G$ leaves $\mu$ invariant.

Compare \cite[Lemma~9.23]{BF} to the following lemma where we include the case $\infty\in P_{N_p}$, the postcritical set of $N_p$.

\begin{lem}\label{Lem:Lem}
	The $G$ invariant $\mu$ defined above satisfies the hypothesis of David Integrability Theorem (Theorem~\ref{Thm:David_Integration}).
\end{lem}
\begin{xrem}
The proof we provide is slightly different than the one given in \cite[Lemma~9.23]{BF}. Our proof is applicable even if there is a critical orbit landing at the repelling fixed point at $\infty$ for PCF Newton maps of polynomials.
\end{xrem}
\begin{proof}[Proof of Lemma~\ref{Lem:Lem}]
	Let $V$ be a simply connected linearizable neighborhood of $\infty$. Let $U$ be a connected component of $N_p^{-1}(V)$ compactly contained in $V$. Let $S_\infty=\psi^{-1}(S'_B)$, where $\psi=\psi_{\cA^\circ}:\cA^\circ\to\D$ is the Riemann map defined above, and $S'_B$ is as in Lemma~\ref{Lem:Parabolic_Surgery_Disk}. Set $\rho={N^{\prime}_p(\infty)}=\frac{d}{d-1}>1$, and let $K_\mu$ be the real dilatation of $\mu$; by Koebe's Distortion Theorem and Lemma \ref{Lem:Parabolic_Surgery_Disk} there exists a constant $C_1>0$ such that for $m$ large enough,

\begin{equation}
\label{Eq:Koebe1}
	\text{Area}\left\{ z\in S_\infty\ |\ K_\mu>m \right\}\le \left(C_1/\rho^{2m}\right)\text{Area}\left\{ S_\infty\right\}.
\end{equation}
Moreover, for $m$ large enough we have 
\begin{equation}
\label{Eq:Subset}
\left\{ z\in S_\infty\ |\ K_\mu>m \right\}\subset U.
\end{equation} 
From now on we assume that $m$ is large enough that \eqref{Eq:Subset} is satisfied. If $N^{\circ k}_p(y)=\infty$ for some $y\in \C$ and $k\ge 1$, then let $S_y$ be a connected component of $N_p^{-\circ k}(S_\infty)$ with vertex at $y$. Here we need to be cautious: if the orbit of $y$ contains a critical point then we end up having more than one component with a shared vertex at $y$; in total, we have $\deg(N_p^{\circ k},y)$ components, the local degrees are uniformly bounded above by $d-2$, the number of free critical points of $N_p$. Let us fix $k\ge1$, and let $U'$ and $V'$ be connected components of $N_p^{-\circ k}(U)$ and $N_p^{-\circ k}(V)$ respectively such that $y\in U'\Subset V'$. 
The maps $N_p^{\circ k}:(V',y)\to(V,\infty)$ and $N_p^{\circ k}:(U',y)\to(U,\infty)$ are coverings, both ramified at $y$ if the orbit of $y$ contains a critical point of $N_p$. As $N_p:U\to V$ fixes $\infty$, its lift is a repelling germ. To obtain the lift, we fix base points $w_0$ for $U\setminus \{\infty\}$ and its image $N_p(w_0)$ for $V\setminus\{\infty\}$ such that both $w_0$ and $N_p(w_0)$ belong to $S_\infty$. Pick a component $S'_y$ of $N_p^{-\circ k}(S_\infty\cap U)$ that is included in $U'$. Let $y_0$ and $z_0$ be base points for $U'\setminus \{y\}$ and  $V'\setminus \{y\}$ respectively such that  $y_0\in N_p^{-\circ k}(w_0)\cap S'_y$ and $z_0\in N_p^{-\circ k}(N_p(w_0))\cap S'_y$. Once the component $S'_y$ is fixed there is only one choice for $y_0$ and $z_0$. As $N_p$ is one-to-one on $U$, the induced homeomorphisms at the level of fundamental groups of involved domains are isomorphic, and thus we can apply Lemma \ref{Lem:Lifting} to the pair $N_p\circ N^{\circ k}_p: \left(U'\setminus \{y\},y_0\right)\to \left(V\setminus \{\infty\},N_p(w_0)\right)$ and $N^{\circ k}_p: \left(V'\setminus \{y\},z_0\right)\to \left(V\setminus \{\infty\},N_p(w_0)\right)$, obtaining a lift \[L:(U'\setminus \{y\},y_0)\to (V'\setminus \{y\},z_0)\] with $L(y_0)=z_0$ such that the following diagram commutes:

  \[\begin{diagram}[heads=LaTeX,l>=3em,PS]
  \left(U'\setminus \{y\},y_0\right) &\rTo^{L} & \left(V'\setminus \{y\},z_0\right)\\
  \dTo^{N_p^{\circ k}} & & \dTo_{N_p^{\circ k}}\\
  \left(U\setminus \{\infty\},w_0\right) &\rTo^{N_p} & \left(V\setminus \{\infty\},N_p(w_0)\right)
  \end{diagram}\]
The lift $L$ is a one-to-one conformal map and satisfies $N_p^{\circ k}\circ L=N_p^{\circ (k+1)}$. We extend $L$ to $y$ by Riemann removability theorem, and it is easy to see that $L(y)=y$. Different choices of $S'_y$, and thus different base points, produce lifts that differ from each other by a composition of a deck transformation of the covering map $N^{\circ k}_p:  \left(V'\setminus \{y\}\right)\to \left(V\setminus \{\infty\}\right)$. Easy calculation shows that $(L'(y))^{k_0}=N'_p(\infty)=\rho$, where $k_0= \deg(N_p^{\circ k},y)\le d-2$. By Koebe's Distortion Theorem and Lemma \ref{Lem:Parabolic_Surgery_Disk}, there exists a constant $C_2>0$ such that
\begin{align}
 \text{Area}\left\{ z\in S_y\ |\ K_\mu>m \right\}&\le \left(C_2/\rho^{2m/k_0}\right)\text{Area }S_y \nonumber\\
\label{Eq:Koebe2}&\le \left(C_2/\rho^{2m/(d-2)}\right)\text{Area }S_y,
\end{align} 
	for all preimages $y\in\C$ of $\infty$. The constant $C_2$ depends only on the multiplier (distortion of $L$) at the fixed point $y$, thus there are only finite number of possibilities (at most $d-2$). Let us denote the maximum of these constants by $C_2$ again. From \eqref{Eq:Koebe1} and \eqref{Eq:Koebe2}, we deduce that for $m$ large enough,
  \begin{align*}
  	\text{Area}\left\{ z\in\Chat\ |\ K_\mu>m \right\} & = \text{Area}\left\{ z\in S_\infty\ |\ K_\mu>m \right\}\\
  	&+\sum\limits_{k\ge 1}\sum\limits_{N_p^{\circ k}(y)=\infty}\text{Area}\left\{ z\in S_y,\  K_\mu>m \right\}\\
& \le   \left(C_1/\rho^{2m}\right)\text{Area}\left\{ S_\infty\right\}\\
&+ \sum\limits_{k\ge 1}\sum\limits_{N_p^{\circ k}(y)=\infty}\left(C_2/\rho^{2m/(d-2)}\right)\text{Area }S_y. \\
 	\end{align*}
 In the above summands in $y\in\C$ with $N_p^{\circ k}(y)=\infty$ all $\deg(N_p,y)$ different $S_y$ are included if $y$ is a critical point. Hence,
 \[\text{Area}\left\{ z\in\Chat\ |\ K_\mu>m \right\}\le\left(C/\rho^{2m/(d-2)}\right)\text{Area}\ X=C'e^{-2m\ln\rho/(d-2)},
  \]
where $C=\max(C_1,C_2)$, $C'=C\text{Area}X$, and $X=S_\infty\cup \bigcup_{k\ge1,N_p^{\circ k}(y)=\infty}S_y$ is the total domain under the consideration. Since the area of $X$ depends on which metric we are using, if we want to use the Euclidean metric then we have to make sure that $\infty$ belongs to the interior of the complement of $X$ in $\Chat$. This can be done by conjugating $N_p$ by a M\"obius map in the beginning of the whole construction so that the set $X$ stays bounded with a finite area. For instance, by a suitable conjugation we can put $\infty$ into some immediate basin. Alternatively, we can use a spherical metric in $\Chat$; then the area of $X$ is finite again. This proves the estimate on $K_\mu$ and finishes the proof of the lemma.
\end{proof}

The David Integrability Theorem (Theorem~\ref{Thm:David_Integration}) asserts the existence of a map $\phi$ which integrates $\mu$. Now we obtain a holomorphic function from of the construction.

\begin{lem}[Straightening]\cite[Lemma 9.25]{BF}
	\label{Lem:Straightening_Haissinsky}
	If there is a David homeomorphism $\phi:\Chat\to\Chat$ such that the equation $\partial_{\bar{z}}\psi=\mu_\phi \partial_z \psi$ has a local solution, unique up to post-composition by a conformal map, and such that $G^*\mu_\phi=\mu_\phi$ a.e., then $R:=\phi\circ G\circ \phi^{-1}$ is a rational function.
\end{lem}

\begin{proof}
	Indeed, on every disk where $G$ is injective, $\phi\circ G$ has the same Beltrami coefficient as $\phi$; hence, there exists a conformal map, say $R$, such that $\phi\circ G=R\circ \phi$, which follows from uniqueness of the solution. There is a discrete set of points, which is the set of critical points of $G$, where $G$ is not locally injective; hence by Riemann removability theorem, $R$ is a rational function. If our map was a polynomial then it can be shown that the result is again a polynomial of the same degree.
\end{proof}

\section{Proof of Main Theorem~\ref{Thm:Parabolic_Surgery_Newton}}
In this section, relying on the surgery construction of previous section, we prove our main result.
\begin{proof}[Proof of Main Theorem~\ref{Thm:Parabolic_Surgery_Newton}] Let $N_p$ be a postcritically finite Newton map of degree $d\ge 3$, and $\Delta^+_n$ be its marked channel diagram with $1\le n\le d$. For all $1\le j\le n$, let $\cA(\xi_j)$ be the marked basins of superattracting fixed points $\xi_j$. Applying Lemma~\ref{Lem:GettingPCM} to $N_p$ through its marked basins and we obtain a Newton map $N_{\hat p}$ and a quasiconformal map $\phi_0$ such that $\phi_0\circ N_p=N_{\hat p}\circ\phi_0$ except on a compact set. Every marked immediate basin of $N_{\hat p}$ becomes (non-super)attracting with a single critical point in it, and $N_{\hat p}$ is postcritically minimal on these marked basins, meaning that its critical orbits are in minimal critical orbit relations. Note that during the surgery construction of the previous section, we denoted the result again by $N_p$. Now we need to go back to the original notation, $\phi_0$ relates $N_p$ and $N_{\hat p}$. 

The last lemma of the previous section yields a David homeomorphism $\phi_1$ and $R$ such that $\phi_1\circ G=R\circ \phi_1$, where $G$ was a topological model. By construction, $\infty$ is a parabolic fixed point of $R$ with the multiplicity $+1$, and finite fixed points are superattracting, since $\phi_1$ is conformal away from marked basins; hence $R$ is a Newton map, say $N_{\tilde p e^{\tilde q}}$. All finite fixed points of $R$ are superattracting and there are $d-n$ of them. Thus $\infty$ is a parabolic fixed point with multiplicity $n+1$, implying that $\deg(\tilde q)=n$. Conjugation preserves dynamics, and in fact by construction $G$ satisfies minimal critical orbit relations, hence $R=N_{\tilde p e^{\tilde q}}$ is a \emph{postcritically minimal} Newton map. The homeomorphisms $\phi_0$ and $\phi_1$ are conformal except on the corresponding marked basins. Set $\phi=\phi_1\circ\phi_0$. Finally, we have $\phi\circ N_p=N_{\tilde p e^{\tilde q}}\circ\phi$ away from marked immediate basins of $N_p$. Moreover, $\phi$ is conformal on the interior of the complement of marked full basins of $N_p$, unmarked components of the Fatou set. Moreover, by the construction of the previous section, the marked invariant accesses of the marked channel diagram $\Delta^+_n$ of $N_p$ became dynamical accesses of the parabolic basin of $\infty$ for $N_{\tilde p e^{\tilde q}}$. Thus, all conditions of Main Theorem~\ref{Thm:Parabolic_Surgery_Newton} are fulfilled for $\phi$ and $N_{\tilde p e^{\tilde q}}$, finishing the proof.
\end{proof}

\subsection*{Acknowledgements}
The author thanks Dierk Schleicher and the dynamics group in Bremen, especially Russell Lodge and Sabyasachi Mukherjee, for their comments that helped to improve this paper. The author would like to thank an anonymous referee for constructive comments and suggestion for improvement.

\end{document}